\documentclass{amsart}

\usepackage{graphicx}                
\usepackage[tight,center]{subfigure} 
\usepackage{enumerate}               

\newtheorem{theorem}{Theorem}
\newtheorem{corollary}[theorem]{Corollary}
\newtheorem{lemma}[theorem]{Lemma}
\newtheorem{algorithm}[theorem]{Algorithm}

\theoremstyle{definition}
\newtheorem{defn}[theorem]{Definition}

\newcommand{\cell}{\mathcal{C}}
\newcommand{\mfd}{\mathcal{M}}
\newcommand{\mfdjr}{\mathcal{M}_\mathrm{JR}}
\newcommand{\ppirr}{{$\mathbb{P}^2$-ir\-re\-du\-ci\-ble}}

\newcommand{\R}{\mathbb{R}}
\newcommand{\regina}{\emph{Regina}}
\newcommand{\tri}{\mathcal{T}}
\newcommand{\trijr}{\mathcal{T}_\mathrm{JR}}
\newcommand{\twisted}{\stackrel{\smash{\protect\raisebox{-1mm}[0pt][0pt]%
    {$\scriptstyle\sim$}}}{\times}}
\newcommand{\Z}{\mathbb{Z}}

\begin{document}

\title[A new approach to crushing 3-manifold triangulations]%
    {A new approach to crushing \\ 3-manifold triangulations}
\author{Benjamin A.~Burton}
\address{School of Mathematics and Physics \\
    The University of Queensland \\
    Brisbane QLD 4072 \\
    Australia}
\email{bab@maths.uq.edu.au}
\thanks{This is the full journal version of a conference paper that appeared in
    \textit{SCG ’13: Proceedings of the Twenty-Ninth Annual Symposium on
    Computational Geometry}, ACM, 2013.}
\thanks{The author is supported by the Australian Research Council
    under the Discovery Projects funding scheme (project DP1094516).}


\begin{abstract}
    The crushing operation of Jaco and Rubinstein is a powerful
    technique in algorithmic 3-manifold topology: it enabled the first
    practical implementations of 3-sphere recognition and prime
    decomposition of orientable manifolds, and it plays a prominent role
    in state-of-the-art algorithms for unknot recognition and testing
    for essential surfaces. Although the crushing operation will always
    reduce the size of a triangulation, it might alter its topology, and
    so it requires a careful theoretical analysis for the settings in
    which it is used.

    The aim of this short paper is to make the crushing operation more
    accessible to practitioners, and easier to generalise to new
    settings. When the crushing operation was first introduced, the
    analysis was powerful but extremely complex. Here we give a new
    treatment that reduces the crushing process to a sequential
    combination of three ``atomic'' operations on a cell decomposition,
    all of which are simple to analyse. As an application, we generalise
    the crushing operation to the setting of non-orientable 3-manifolds,
    where we obtain a new practical and robust algorithm for
    non-orientable prime decomposition. We also apply our crushing
    techniques to the study of non-orientable minimal triangulations.
\end{abstract}

\maketitle

%
%

\section{Introduction}

Algorithms in computational 3-manifold topology often exhibit an enormous gap
between theory and practice.
Theoretical solutions are now known for a large number
of difficult 3-dimensional topological problems, ranging from smaller problems
such as recognising the unknot \cite{haken61-knot} or recognising the 3-sphere
\cite{rubinstein95-3sphere} through to decomposition into geometric
pieces \cite{jaco95-algorithms-decomposition}
and of course the full homeomorphism (or ``topological equivalence'') problem
\cite{haken62-homeomorphism,jaco05-lectures-homeomorphism,matveev03-algms}.
Although these results are of high significance to the
mathematical community, many remain
algorithms in theory only---such algorithms are often far too intricate
to implement, and far too slow to run.

In the last decade, however, there has been strong progress in the realm
of practical, usable algorithms on 3-manifolds.
For instance, there are now practical implementations
of unknot recognition, 3-sphere recognition and orientable prime decomposition
\cite{andreeva02-webservice,burton04-regina}, and recently more
complex algorithms such as testing for closed essential surfaces 
have become viable \cite{burton13-large,burton12-ws}.

A key component in many of these practical algorithms is the
\emph{crushing procedure} of Jaco and Rubinstein \cite{jaco03-0-efficiency}.
This procedure was developed as part of their theory of \emph{0-efficiency},
drawing on earlier unpublished work of Andrew Casson.
It operates in the context of normal surface theory,
a common algorithmic toolkit for 3-manifold topologists.
In essence, the crushing process modifies a triangulation to
eliminate ``unwanted'' normal spheres and discs, whereupon the resulting
triangulation is called \emph{0-efficient}
(we give a more precise definition shortly).  This brings
both theoretical and practical advantages: 0-efficient
triangulations are typically smaller and easier to study,
and algorithms upon them are easier to formulate---often significantly so
\cite{jaco03-0-efficiency,jaco03-decision,li11-genus,rubinstein04-smallsfs}.

Although the full process of obtaining a 0-efficient triangulation
requires worst-case exponential time, recent techniques based on combinatorial
optimisation have made this extremely fast in a range of experimental
settings \cite{burton13-large,burton12-unknot}.
A notable application of crushing has been in 3-sphere
recognition: here the introduction of Jaco and Rubinstein's 0-efficiency
techniques was a major turning point that
made 3-sphere recognition practical to implement for the first time
\cite{burton13-regina,jaco03-0-efficiency}.

In summary:
normal surface theory makes difficult 3-manifold
problems \emph{decidable}, whereas crushing and 0-efficiency often
play a key role in making the resulting algorithms \emph{practical}.
It is therefore important for practitioners in computational 3-manifold
topology to understand crushing and
0-efficiency, and to be able to apply them to new settings.
The aims of this paper are (i)~to make the crushing operation more
accessible to
the wider computational topology community, (ii)~to simplify its analysis
so that the techniques are easier to use and generalise, and
(iii)~to apply this simplified analysis to the non-orientable setting,
yielding a new practical and robust algorithm for non-orientable
prime decomposition.

In detail: the crushing procedure
eliminates unwanted normal spheres and discs from
a triangulation by cutting the manifold open along them, collapsing
the resulting spheres or discs on the boundary to points, and then further
``flattening'' the resulting cell decomposition until we once again obtain a
(different) triangulation.  Importantly, this crushing procedure
(i)~is simple to implement, and (ii)~always simplifies the triangulation
by reducing the number of tetrahedra (neither of which are true for the
related operation of cutting along a normal surface and
retriangulating).  The downside is that crushing could change
the topology of the underlying 3-manifold in unintended ways, and so this
crushing process is ``destructive'';
however, the possible changes are often both simple and detectable.

One difficulty with Jaco and Rubinstein's original paper is that,
although their techniques are extremely powerful, the accompanying
analysis is extremely complex: they study the potential effects of the
crushing procedure through a series of detailed arguments
as they collapse chains of truncated prisms and product regions
throughout the triangulation.
A second difficulty is that their analysis is restricted
to orientable 3-manifolds only.

In Section~\ref{s-crush} of this paper we both
simplify and generalise these arguments.
The key result is Lemma~\ref{l-crush} (the \emph{crushing
lemma}), which shows that---after the initial act of cutting along
and collapsing the
original normal surface---the entire Jaco-Rubinstein
crushing procedure can be expressed
as a sequential combination of three local
atomic operations on a cell decomposition:
flattening a triangular or bigonal pillow to a face,
and flattening a bigon face to an edge.
Therefore, to analyse the ``destructive'' consequences of crushing
in any given setting,
we merely need to examine what can happen independently under
each of these atomic operations.
All three operations are simple to analyse:
Lemma~\ref{l-atomic} lists the possible consequences of each operation,
and Corollary~\ref{c-jrcrush} packages these together to describe
the overall effects of the full crushing process.

We emphasise that these results are general.
We never assume orientability, and all results apply to compact manifolds
both with or without boundary.
Moreover, the key crushing lemma applies equally well to
\emph{ideal triangulations}, which triangulate non-compact manifolds
by allowing vertices whose links are higher genus surfaces.
The analysis of atomic operations
in Lemma~\ref{l-atomic} and Corollary~\ref{c-jrcrush}
is also
straightforward in the ideal case, but the consequences of crushing
become more numerous, and so in this short paper we restrict
this latter analysis to triangulations of compact manifolds only.

In Section~\ref{s-app} we apply our results to
develop the first practical algorithm for computing the prime
decomposition of 3-manifolds
to encompass both the orientable and non-orientable cases.\footnote{%
    Recall that \emph{prime decomposition} asks us to decompose a given
    3-manifold into a connected sum of prime 3-manifolds.
    The \emph{connected sum} $M \# N$ of two manifolds
    $M$ and $N$ is formed by removing a small
    ball from each summand and gluing the summands together along
    the resulting sphere boundaries.}

For orientable manifolds, a modern implementation of prime
decomposition works by repeatedly crushing away normal spheres using the
Jaco-Rubinstein procedure, and then ``reading off'' prime summands from the
resulting collection of disconnected triangulations (some summands will
have disappeared but we can restore these using homology).
It is simple to discard trivial (3-sphere) summands, since the
efficiency-based 3-sphere recognition algorithm dovetails into this
procedure naturally.
The blueprint for this algorithm is laid out in the original
0-efficiency paper \cite{jaco03-0-efficiency}; see \cite{burton13-regina}
for a modern ``ready to implement'' version.

For non-orientable manifolds, however, the current situation is much worse:
the only available algorithm is the older Jaco-Tollefson method
\cite{jaco95-algorithms-decomposition}, where we must
build a collection of disjoint embedded 2-spheres within the input
triangulation using a complex series of cut-and-paste operations,
and then cut along these
2-spheres and retriangulate to obtain the individual summands
(an expensive operation that could vastly increase the
number of tetrahedra).  Detecting trivial summands is also
significantly more complex to implement in this setting.

In Section~\ref{s-app} we bring the non-orientable algorithm in line
with its simpler orientable cousin:
using the new generalised results of Section~\ref{s-crush}, we
show that one can crush away normal spheres and then
``read off'' the summands (again restoring missing summands via homology).
There is an important complication:
if the input manifold contains an embedded two-sided
projective plane then the Jaco-Rubinstein crushing process could fail.
We show that even in this setting, we can still run the algorithm:
it might still succeed, and if it does fail due to a two-sided projective
plane then we obtain a simple certificate alerting us to this fact
(this is the sense in which the algorithm is ``robust'').

We finish in Section~\ref{s-app-minimal} with a second application of
our crushing techniques, this time to the study of minimal
triangulations.  In particular, we show that minimal triangulations of
closed {\ppirr} 3-manifolds must be 0-efficient (modulo a handful
of exceptions), and as a result we identify several constraints on the
combinatorial structures of such triangulations.
Most of these constraints are already known, but
past proofs have relied on ad-hoc arguments;
the purpose of Section~\ref{s-app-minimal} is to show that they
all follow directly from 0-efficiency.  Overall, this discussion on
minimal triangulations is a simple but informative extension
of Jaco and Rubinstein's original discussion in the orientable case
\cite{jaco03-0-efficiency}.

Beyond its theoretical contributions, the results of this paper
are important for practitioners.  In particular, the non-orientable
prime decomposition algorithm of Section~\ref{s-app}
will soon appear in the software package {\regina} \cite{regina}.

For the special case of closed orientable manifolds, Fowler describes
a different approach to simplifying 0-efficiency arguments using
\emph{spines},
elaborating on an earlier argument of Casson \cite{fowler03-0-efficient}.
See also Matveev's book \cite{matveev03-algms}, which includes a more general
discussion of cutting along
normal surfaces in the setting of special and almost special spines.


\subsection{Preliminaries}

As is common in computational 3-manifold topology, we work not with
simplicial complexes but with smaller and more flexible structures.
A \emph{generalised triangulation} $\tri$
is defined to be a collection of $n$ abstract tetrahedra,
some or all of whose $4n$ faces are affinely identified (or ``glued
together'')
in pairs.  The underlying topological space is often (but not always) a
3-manifold $\mfd$, in which case we say that $\tri$ \emph{triangulates}
$\mfd$.

In a generalised triangulation, we allow two faces of the same
tetrahedron to be identified.  Moreover, as a consequence of the face
identifications, we might find that several edges of a tetrahedron
become identified, and likewise with vertices.
The \emph{link} of a vertex $V$ in a triangulation is the
surface obtained as the frontier of a
small regular neighbourhood of $V$.

For convenience, we define several useful (and mutually exclusive)
subclasses of generalised triangulations.
A \emph{closed triangulation} is one that triangulates
a closed 3-manifold (here every tetrahedron face must be glued to some
partner, and every vertex link must be a sphere).
A \emph{bounded triangulation} is one that triangulates a
compact 3-manifold with non-empty boundary (here one or more tetrahedron
faces are left unglued, and all vertex links must be spheres or discs).
An \emph{ideal triangulation} is one in which every tetrahedron face
is glued to some partner, but some vertices have links that are not
spheres (e.g., tori, Klein bottles, or other higher-genus surfaces).
Ideal triangulations are used
to represent non-compact 3-manifolds by removing the tetrahedron vertices;
a famous example is Thurston's 2-tetrahedron
ideal triangulation of the figure eight knot complement 
\cite{thurston78-lectures}.

In this paper we modify triangulations to obtain more general
\emph{cell decompositions}.
Informally, these are natural extensions of
generalised triangulations that allow 3-cells other than tetrahedra.
A cell decomposition begins with
a collection of abstract 3-cells, which are topological 3-balls
whose boundaries are decomposed into curvilinear polygonal faces
(in particular, we allow small 2-faces such as bigons,
and we allow small 3-cells that are ``pillows'' bounded by a
pair of opposite 2-faces).
Moreover, we endow each edge on the boundary of each 3-cell with an
affine structure (i.e., a homeomorphism from the edge to the interval
$[0,1]$).
We explicitly list the possible types of 3-cell as we encounter them
in this paper, and so we do not go into further detail here.

To form a cell decomposition, we identify (or ``glue'' together) some or
all of the 2-faces of these 3-cells in pairs, using homeomorphisms that
map edges to edges and vertices to vertices, and that restrict to affine
maps on the edges.
This generalises the affine maps between 2-faces that we use for triangulations.
As with triangulations, we allow two
2-faces of the same 3-cell to be identified.
For a concrete example of a cell decomposition, see
Definition~\ref{d-jrcrush} below.

We define an \emph{invalid edge} of a generalised triangulation or
cell decomposition to be one that (as a consequence of the
2-face gluings) becomes identified with itself in reverse.
A triangulation or cell decomposition is called \emph{valid}
if it does not contain an invalid edge.
The underlying topological space of an invalid triangulation or
cell decomposition cannot be a 3-manifold, since a small regular
neighbourhood of the midpoint of an invalid edge will be bounded by $\R P^2$.

A \emph{normal surface} in a generalised triangulation $\tri$
is a properly embedded surface\footnote{%
    A surface $S$ is \emph{properly embedded} in a triangulation $\tri$
    if $S$ has no self-intersections, and the boundary of $S$ is
    precisely where $S$ meets the boundary of $\tri$.}
in $\tri$ that meets each tetrahedron in
a (possibly empty) collection of curvilinear triangles and
quadrilaterals, as illustrated in Figure~\ref{fig-normaldiscs}.
A \emph{vertex linking surface} (also called a \emph{trivial surface})
is a connected normal surface formed entirely from triangles;
any such surface must surround some vertex $V$ of
the triangulation as illustrated in Figure~\ref{fig-vertexlink},
and effectively triangulates the link of $V$.

\begin{figure}[tb]
\centering
\subfigure[Triangles and quadrilaterals]{%
    \label{fig-normaldiscs}%
    \hspace{1cm}\includegraphics[scale=0.5]{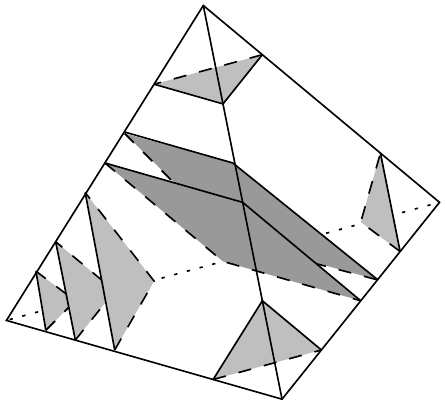}\hspace{1cm}}
\hspace{2cm}
\subfigure[A vertex linking surface]{%
    \label{fig-vertexlink}%
    \includegraphics[scale=0.8]{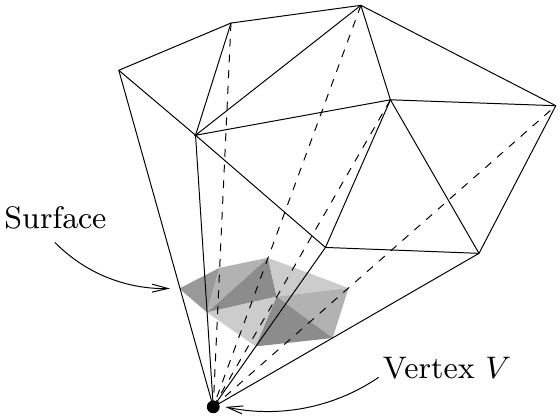}}
\caption{Normal surfaces within a triangulation}
\end{figure}

The concept of \emph{0-efficiency} is defined as follows
\cite{jaco03-0-efficiency}.
If $\tri$ is a closed or ideal triangulation, then we call $\tri$
0-efficient if and only if it contains no non-trivial normal spheres.
If $\tri$ is a bounded triangulation, then we call $\tri$ 0-efficient
if and only if it contains no non-trivial normal discs.%
    \footnote{Jaco and Rubinstein show that, under appropriate assumptions,
    if $\tri$ contains no non-trivial normal discs then
    $\tri$ must contain no non-trivial normal spheres
    also \cite[Proposition~5.15]{jaco03-0-efficiency}.}

Jaco and Rubinstein describe a general ``destructive'' crushing
procedure which they use for many purposes, such as creating 0-efficient
triangulations of manifolds, and decomposing orientable manifolds into
connected sums.  This procedure is the main focus of
this paper, and we describe it now in detail.

\begin{defn} \label{d-jrcrush}
    Let $S$ be a normal surface in some generalised triangulation $\tri$.
    The \emph{Jaco-Rubinstein crushing procedure} operates on $S$ as
    follows:
    \begin{enumerate}
        \item \label{en-jrcrush-cut}
        We cut $\tri$ open along the normal surface $S$.
        This converts the triangulation into a cell decomposition
        with a large variety of possible cell types (such as truncated
        tetrahedra, triangular or quadrilateral prisms, and truncated
        triangular prisms, some of which are
        illustrated in Figure~\ref{fig-jrslices}).
        If $S$ is two-sided in $\tri$
        then we obtain two new copies of $S$ on the
        boundary of this cell decomposition,
        and if $S$ is one-sided in $\tri$ then we obtain one new copy of
        the double cover of $S$ on the boundary.

        \begin{figure}[t]
            \centering
            \includegraphics[scale=0.7]{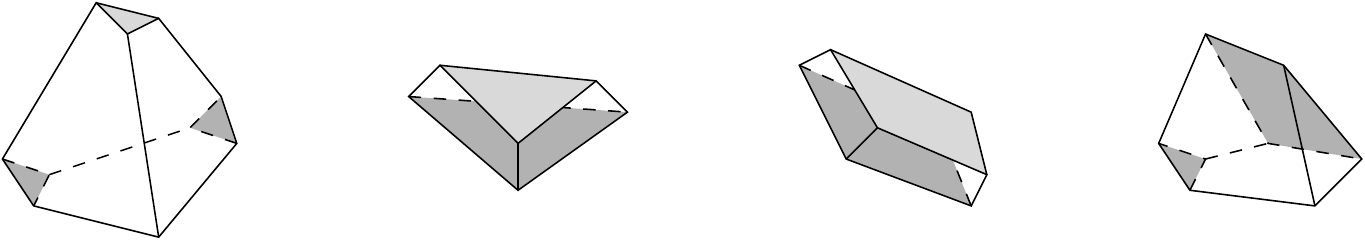}
            \caption{Examples of cells obtained after cutting open along $S$}
            \label{fig-jrslices}
        \end{figure}

        \begin{figure}[t]
            \centering
            \includegraphics[scale=0.7]{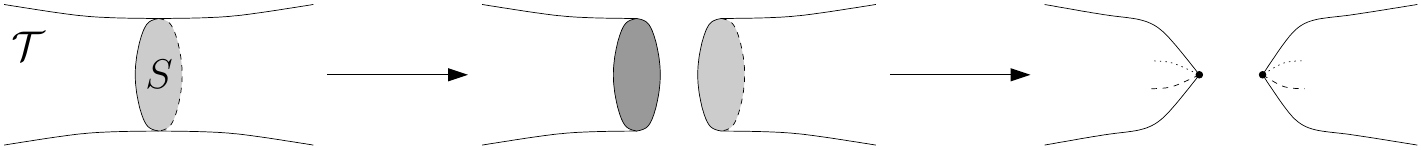}
            \caption{Collapsing copies of $S$ on the boundary to points}
            \label{fig-crushbdry}
        \end{figure}

        \item \label{en-jrcrush-collapse}
        We then collapse (or ``shrink'')
        each copy of $S$ on the boundary to a point
        (using the quotient topology), as illustrated in
        Figure~\ref{fig-crushbdry}.
        Specifically, if $S$ was two-sided in $\tri$ then we collapse the two
        copies of $S$ on the boundary to two points, and if $S$ was
        one-sided in $\tri$ then we collapse the double cover of $S$ on
        the boundary to one point.
        This converts the
        triangulation into a cell decomposition $\cell$,
        with cells of the following types:
        \begin{itemize}
            \item \emph{3-sided footballs},
            illustrated in Figure~\ref{fig-jr3football},
            which we obtain from regions of $\tri$ between two
            parallel triangles of $S$,
            or between a triangle of $S$ and a tetrahedron vertex;
            \item \emph{4-sided footballs},
            illustrated in Figure~\ref{fig-jr4football},
            which we obtain from regions of $\tri$ between two
            parallel quad\-rilaterals of $S$;
            \item \emph{triangular purses},
            illustrated in Figure~\ref{fig-jrpurse},
            which we obtain from regions of $\tri$
            between a quadrilateral of $S$
            and nearby triangles or tetrahedron vertices;
            \item \emph{tetrahedra},
            which we obtain from the central regions of tetrahedra in $\tri$
            that do not contain any quadrilaterals of $S$.
        \end{itemize}

        \begin{figure}[tb]
        \centering
        \subfigure[A 3-sided football]{\label{fig-jr3football}%
            \includegraphics[scale=0.8]{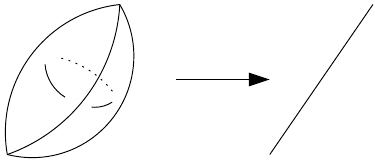}}
        \hspace{0.8cm}
        \subfigure[A 4-sided football]{\label{fig-jr4football}%
            \includegraphics[scale=0.8]{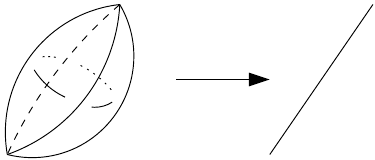}}
        \hspace{0.8cm}
        \subfigure[A triangular purse]{\label{fig-jrpurse}%
            \includegraphics[scale=0.8]{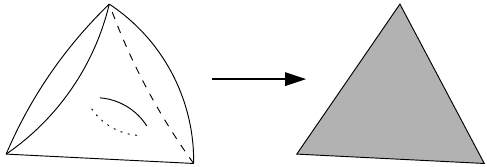}}
        \caption{Destructively flattening non-tetrahedron cells}
        \label{fig-jrcrush}
        \end{figure}

        \item \label{en-jrcrush-flatten}
        We next eliminate any non-tetrahedron cells, as
        illustrated in Figure~\ref{fig-jrcrush}, by simultaneously
        flattening all footballs to edges, and flattening all triangular
        purses to triangular faces (again using the quotient topology).
        Note that all remaining tetrahedron cells are preserved in this step
        (the flattening operations only affect cells with bigon faces).

        We can now ``read off'' a resulting generalised triangulation
        $\trijr$, which is defined \emph{only} by the surviving tetrahedra
        and the resulting identifications between their 2-dimensional faces.
        In particular:
        \begin{itemize}
            \item Any triangles, edges and/or vertices that do not belong to
            a tetrahedron are removed entirely,
            as illustrated in Figure~\ref{fig-clean-edgeface}.
            We might even lose entire connected components in this way.
            \item If different pieces of a triangulation are connected along
            pinched edges or vertices
            then these pieces will ``fall apart'', as illustrated in
            Figure~\ref{fig-clean-pinch} (since there are no 2-dimensional
            faces holding them together).
        \end{itemize}
    \end{enumerate}

    \begin{figure}[tb]
    \centering
    \subfigure[Edges or faces not in a tetrahedron will disappear]%
        {\label{fig-clean-edgeface}%
        \qquad\includegraphics[scale=0.55]{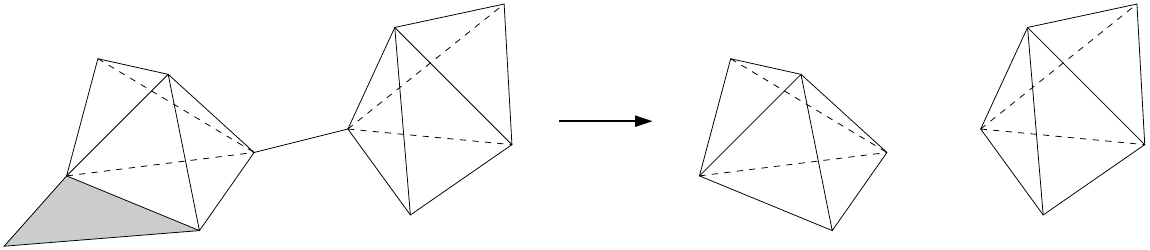}\qquad}
    \subfigure[Tetrahedra joined along pinched edges or vertices
        will fall apart]%
        {\label{fig-clean-pinch}%
        \qquad\qquad\includegraphics[scale=0.55]{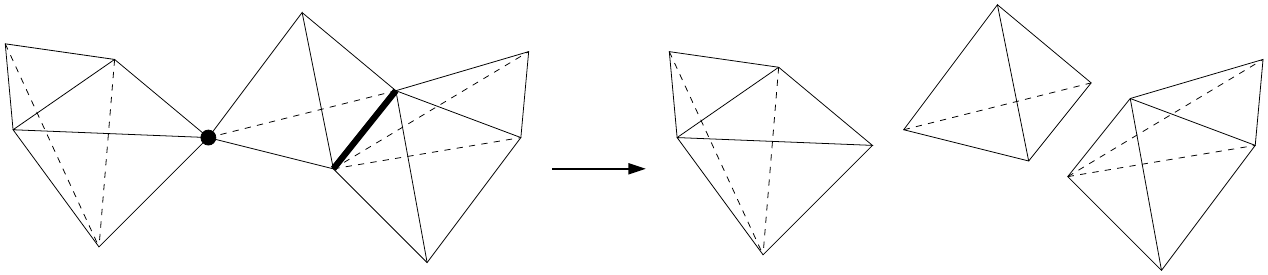}\qquad\qquad}
    \caption{Triangulations are defined by face gluings only}
    \label{fig-clean}
    \end{figure}

    The final result of the crushing procedure
    is this generalised triangulation $\trijr$.
    Note that this might be disconnected, empty or invalid,
    and might not even represent a 3-manifold.

    For convenience, we refer to steps~\ref{en-jrcrush-cut} and
    \ref{en-jrcrush-collapse} as \emph{non-destructive crushing} (yielding
    the cell complex $\cell$), and
    all three steps~\ref{en-jrcrush-cut}--\ref{en-jrcrush-flatten}
    as \emph{destructive crushing} (yielding the generalised
    triangulation $\trijr$).  Unless otherwise specified,
    ``crushing'' always refers to the full destructive operation.
\end{defn}

We observe
that each original tetrahedron $\Delta$ of $\tri$ can only give rise to at
most one tetrahedron of $\trijr$ (and only if $\Delta$
contains no quadrilaterals of $S$).  It follows that $\trijr$ has
strictly fewer tetrahedra than $\tri$, unless $S$ is a union of vertex linking
surfaces, in which case $\trijr$ and $\tri$ are isomorphic (i.e., the
crushing procedure has no effect).

Note that the intermediate cell decomposition $\cell$ and the
final triangulation $\trijr$ are well-defined. In particular,
the flattening operations above can never result in any 2-face of a
cell being identified with more than one partner 2-face,
and the final triangulation $\trijr$ is independent of any
``order'' in which we flatten the footballs and/or purses of $\cell$.

A key property of this procedure (which we generalise in
Corollary~\ref{c-jrcrush}) is that,
for orientable manifolds, any ``destructive'' changes are both limited
and detectable:

\begin{theorem}[Jaco and Rubinstein \cite{jaco03-0-efficiency}]%
    \label{t-jrcrushor}
    Let $\tri$ be a generalised triangulation of a compact
    orientable 3-manifold $\mfd$ (with or without boundary),
    and let $S$ be a normal sphere or
    disc in $\tri$.  Then, if we destructively crush $S$ using the
    Jaco-Rubinstein procedure, we obtain a valid generalised triangulation
    $\trijr$ whose underlying 3-manifold $\mfdjr$ is obtained from
    $\mfd$ by zero or more of the following operations:
    \begin{itemize}
        \item undoing connected sums, i.e.,
        replacing some intermediate manifold $\mfd'$ with the disjoint union
        $\mfd'_1 \cup \mfd'_2$, where $\mfd' = \mfd'_1\ \#\ \mfd'_2$;
        \item cutting open along properly embedded discs;
        \item filling boundary spheres with 3-balls;
        \item deleting 3-ball, 3-sphere, $\R P^3$, $L_{3,1}$ or
        $S^2 \times S^1$ components.\footnote{%
            Regarding notation: $\R P^3$ denotes real projective space,
            $L_{3,1}$ is a lens space, and
            $S^2 \times S^1$ is the product space of the 2-sphere and the
            circle.}
    \end{itemize}
\end{theorem}

For reference, Jaco and Rubinstein do not present Theorem~\ref{t-jrcrushor}
in this unified form---the theorem statement above collects the results
of several detailed arguments from
throughout their original paper \cite{jaco03-0-efficiency}.

We emphasise again that all generalised triangulations and
cell decompositions in this paper are defined entirely by their
3-cells and the pairwise identifications between the 2-faces of these 3-cells.
In particular, if we modify a cell decomposition so that some
edge or 2-face does not belong to a 3-cell then that edge or 2-face will
disappear (as in Figure~\ref{fig-clean-edgeface}), and if different pieces of
the cell decomposition become connected along pinched edges or vertices then
those pieces will fall apart (as in Figure~\ref{fig-clean-pinch}).


\section{The crushing lemma} \label{s-crush}

In this section we present our ``atomic'' formulation of the Jaco-Rubinstein
crushing procedure.
We begin with the crushing lemma (Lemma~\ref{l-crush}),
which establishes the sufficiency of our three atomic operations, and
shows that they can be performed sequentially (as opposed to simultaneously).
Lemma~\ref{l-atomic} then analyses the precise behaviour of each operation
on a compact manifold, and Corollary~\ref{c-jrcrush} uses this to
prove a generalisation of Theorem~\ref{t-jrcrushor} that covers both
orientable and non-orientable manifolds.

We emphasise that the crushing lemma is completely general: the
triangulation may be non-orientable, or ideal, or even invalid.
In this sense, the crushing lemma is intended as a launching point
for generalising crushing and 0-efficiency
technology to a wide range of settings
(such as ideal triangulations, which we do not pursue in detail in this
short paper).

The proof of the crushing lemma uses an algorithmic approach: we show
how the full crushing procedure can be performed one atomic operation
at a time.  We note that this algorithm is intended to assist with the
theoretical analysis, not the implementation; a practical implementation
could simply flatten non-tetrahedron cells ``in bulk''.\footnote{%
    The reader is invited to peruse {\regina}'s source code \cite{regina}
    to see how this can be done; see the function
    \texttt{NNormalSurface::\allowbreak crush()}.}

\begin{lemma}[Crushing lemma] \label{l-crush}
    Let $\tri$ be a generalised triangulation containing a normal surface $S$.
    Let $\cell$ be the cell decomposition obtained by
    non-destructively crushing $S$, as described in
    steps~(\ref{en-jrcrush-cut})--(\ref{en-jrcrush-collapse})
    of Definition~\ref{d-jrcrush}, and let
    $\trijr$ be the final triangulation obtained at the end of the
    destructive crushing procedure, after
    flattening away all non-tetrahedron cells in
    step~(\ref{en-jrcrush-flatten}) of Definition~\ref{d-jrcrush}.
    Then $\trijr$ can be obtained from
    $\cell$ by a sequence of zero or more of the following atomic operations,
    one at a time, in some order:
    \begin{itemize}
        \item flattening a triangular pillow to a triangular face,
        as shown in Figure~\ref{fig-jrtripillow};
        \item flattening a bigonal pillow to a bigon face, as shown in
        Figure~\ref{fig-jrbipillow};
        \item flattening a bigon face to an edge, as shown in
        Figure~\ref{fig-jrbigon}.
    \end{itemize}
    As in Definition~\ref{d-jrcrush},
    after each atomic operation we remove any ``orphaned'' 2-faces,
    edges or vertices that do not belong to a 3-cell, and we pull apart any
    pieces of the cell decomposition that are connected along pinched
    edges or vertices
    (as in Figure~\ref{fig-clean}).
\end{lemma}

\begin{figure}[tb]
\centering
\subfigure[Flattening a triangular pillow]{\label{fig-jrtripillow}%
    \includegraphics[scale=0.75]{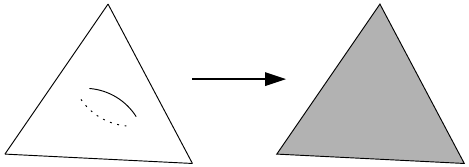}}
\hspace{0.55cm}
\subfigure[Flattening a bigonal pillow]{\label{fig-jrbipillow}%
    \includegraphics[scale=0.75]{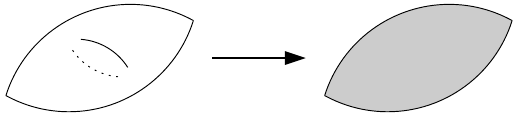}}
\hspace{0.55cm}
\subfigure[Flattening a bigon face]{\label{fig-jrbigon}%
    \includegraphics[scale=0.75]{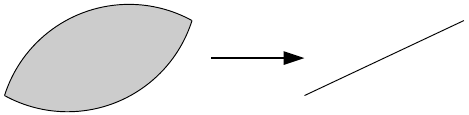}}
\caption{Atomic moves for the Jaco-Rubinstein crushing procedure}
\label{fig-atomic}
\end{figure}

Note that each atomic operation might be performed several times, and that
multiple instances of one operation might be interspersed with
instances of the others.

\begin{proof}
    Recall that in the cell decomposition $\cell$, there are only
    three types of cells that are not tetrahedra:
    \emph{3-sided footballs},
    \emph{4-sided footballs} and
    \emph{triangular purses},
    which $\trijr$ flattens to edges, edges and triangles respectively.
    In addition to these, we now describe three \emph{intermediate}
    cell types which might appear as we incrementally convert
    $\cell$ into $\trijr$:
    \begin{itemize}
        \item \emph{triangular pillows}, as seen earlier in
        Figure~\ref{fig-jrtripillow};
        \item \emph{bigonal pillows}, as seen earlier in
        Figure~\ref{fig-jrbipillow};
        \item \emph{bigonal pyramids}, a new cell type shown in
        Figure~\ref{fig-jrbipyramid}, which $\trijr$ flattens to a
        triangle.
    \end{itemize}
    For reference, all six original and intermediate
    cell types are illustrated in Figure~\ref{fig-jrorigint}.

    \begin{figure}[tb]
    \centering
    \begin{tabbing}
    {}\=
    \subfigure[A 3-sided football]{%
        \includegraphics[scale=0.8]{jr3football}}
    \hspace{1.0cm}
    \=\subfigure[A 4-sided football]{%
        \includegraphics[scale=0.8]{jr4football}}
    \hspace{1.0cm}
    \=\subfigure[A triangular purse]{%
        \includegraphics[scale=0.8]{jrpurse}} \\
    \>\subfigure[A bigonal pyramid]{\label{fig-jrbipyramid}%
        \includegraphics[scale=0.75]{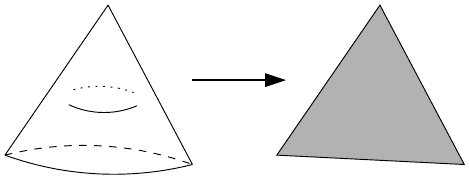}}
    \>\subfigure[A triangular pillow]{%
        \includegraphics[scale=0.75]{jrtripillow}}
    \>\subfigure[A bigonal pillow]{%
        \includegraphics[scale=0.75]{jrbipillow}}
    \end{tabbing}
    \caption{All six original and intermediate cell types}
    \label{fig-jrorigint}
    \end{figure}

    The crushing process requires us to simultaneously
    flatten footballs to edges and purses to triangles;
    our task now is to find a good \emph{order} in which to do this, so
    that we obtain a sequence of atomic operations as described above.
    The key complication is that local moves on one cell might change
    the shapes of adjacent cells, and so we must choose our ordering
    carefully to avoid creating any unexpected new cell types.

    Our solution is to convert $\cell$ into $\trijr$ by applying
    the following algorithm:
    \begin{enumerate}
        \item If there are any triangular pillows, flatten them to triangles.%
        \label{en-crush-start}
        \item If there are any bigonal pillows, flatten them to bigons.
        \item If there is a 3-sided football, then choose one of its
        bigon faces
        which is not identified with another face of the same 3-cell,
        flatten this bigon to an edge (which will also change the
        shape of the adjacent cell, if there is one), and return to
        step~(\ref{en-crush-start}).%
        \label{en-crush-3football}
        \item If there is a 4-sided football, then choose any
        one of its bigon faces, flatten this bigon to an edge,
        and return to step~(\ref{en-crush-start}).%
        \label{en-crush-4football}
        \item If there are any bigonal pyramids or triangular purses,
        flatten all of their bigon faces to edges and
        return to step~(\ref{en-crush-start}).%
        \label{en-crush-pillows}
    \end{enumerate}

    We first note that the choice in step~(\ref{en-crush-3football})
    is always possible because a 3-sided football has an odd number of
    bigon faces.  Moreover, the various flattening operations never
    introduce any new cell types beyond the six listed above:
    \begin{itemize}
        \item In steps~(\ref{en-crush-3football}) and
        (\ref{en-crush-4football}), the 3-sided football becomes a
        bigonal pillow, and the 4-sided football becomes \emph{either}
        a bigonal pillow or a 3-sided football according to whether
        the chosen bigon joins the football with itself or a different cell.
        The adjacent cell (if there is one) changes as follows:
        Because we have already eliminated
        bigonal pillows, the adjacent cell must be a 3-sided football,
        a 4-sided football, a triangular purse, or a bigonal pyramid.
        Flattening the bigon then converts this to a bigonal pillow,
        a 3-sided football,
        a bigonal pyramid, or a triangular pillow respectively.

        \item In step~(\ref{en-crush-pillows}), the only non-tetrahedron
        cell types remaining are triangular pillows with one or two
        bigon sides, and so flattening bigons converts all of these
        cells to triangular pillows.
    \end{itemize}

    We run the algorithm above until there are no
    remaining non-tetrahedron cells.  The algorithm terminates because
    in each iteration it strictly reduces the number of
    non-tetrahedron cells plus the number of bigons.
    The resulting triangulation is then $\trijr$, and the only
    atomic operations that the algorithm performs are
    flattening triangular pillows, bigonal pillows and bigons,
    one at a time.
\end{proof}

One might observe that the crushing lemma simply replaces
the three original moves of Figure~\ref{fig-jrcrush} with the three
atomic moves of Figure~\ref{fig-atomic}.
Nevertheless, this brings important advantages:
\begin{itemize}
    \item The new atomic moves operate on smaller subcomplexes
    (triangular pillows, bigonal pillows and bigon faces), which
    means fewer special cases or unusual behaviours to analyse.
    \item More importantly, our new atomic moves can be performed
    \emph{sequentially}, and can therefore be studied individually as
    \emph{local operations}.
    The original flattening moves of Figure~\ref{fig-jrcrush} must be done
    \emph{simultaneously} (otherwise we introduce many additional
    cell types each with their own moves and analyses),
    which means the original flattening moves must be studied as a
    complex \emph{global operation}
    (as Jaco and Rubinstein do in their original paper).
    \item We extend our analysis to more general settings, such as
    non-orientable triangulations and arbitrary ideal triangulations.
\end{itemize}

From this point onwards we restrict our attention to compact manifolds
(i.e., closed or bounded triangulations), and
study the possible outcomes of each of our three atomic moves.

\begin{lemma} \label{l-atomic}
    Let $\cell$ be a cell decomposition of a compact 3-manifold $\mfd$
    (with or without boundary) that contains no two-sided projective planes.
    Then applying one of the atomic moves of Lemma~\ref{l-crush} will
    yield a (valid) cell decomposition of a 3-manifold $\mfd'$,
    where either $\mfd' = \mfd$, or else $\mfd'$ is obtained from $\mfd$ by
    one of the following operations:
    \begin{itemize}
        \item If we flattened a triangular pillow,
        then $\mfd'$ might remove a single connected
        3-ball, 3-sphere or $L_{3,1}$ component from $\mfd$;
        \item If we flattened a bigonal pillow,
        then $\mfd'$ might remove a single connected
        3-ball, 3-sphere or $\R P^3$ component from $\mfd$;
        \item If we flattened a bigon face, then
        (i)~$\mfd'$ might be obtained by cutting $\mfd$ open
        along a properly embedded disc,
        (ii)~$\mfd'$ might be obtained by filling one boundary sphere of
        $\mfd$ with a 3-ball;
        (iii)~$\mfd'$ might be obtained by cutting $\mfd$ open
        along an embedded sphere and filling the two resulting
        boundary spheres with 3-balls;
        or
        (iv)~we might have $\mfd = \mfd'\ \#\ \R P^3$; that is,
        $\mfd'$ might remove a single $\R P^3$ summand from the
        connected sum decomposition of $\mfd$.
    \end{itemize}
\end{lemma}

\begin{proof}
    The pillow moves are easiest to handle.  Because they do not affect
    the 1-skeleton of $\cell$, it is clear that the resulting cell
    decomposition remains valid.  Let $F_1,F_2$ be the two faces bounding the
    (triangular or bigonal) pillow.
    \begin{itemize}
        \item If $F_1$ and $F_2$ both lie in the boundary of the
        manifold then the pillow
        is an entire 3-ball component, and flattening the pillow simply
        deletes this component.
        \item If $F_1$ and $F_2$ are identified together then the
        pillow is a single connected component of one of the following
        topological types:
        \begin{itemize}
            \item $L_{3,1}$ (obtained by identifying the faces of a
            triangular pillow with a twist);
            \item $\R P^3$ (obtained by identifying the faces of a bigon pillow
            with a twist);
            \item $S^3$ (obtained by identifying the faces of a triangular or
            bigon pillow without a twist).
        \end{itemize}
        Here we ignore orientation-reversing identifications between
        $F_1$ and $F_2$, because these imply that $\cell$ has either an edge
        identified with itself in reverse (i.e., an invalid
        cell decomposition) or a non-orientable vertex link.
        \item Otherwise, $F_1$ and $F_2$ are not identified (so their
        relative interiors are disjoint) and they are not both boundary,
        whereby flattening them together does not change the
        underlying 3-manifold.
    \end{itemize}

    Flattening a bigon face to an edge is a little more delicate, with
    several cases to consider.
    Let $e_1,e_2$ be the two edges bounding the bigon.
    \begin{itemize}
        \item If the entire bigon lies in the boundary of the manifold:
        \begin{itemize}
            \item If $e_1$ and $e_2$ are not identified (so their
            relative interiors are disjoint), then flattening them
            together does not change the underlying 3-manifold.

            \item If $e_1$ and $e_2$ are identified then (since $\mfd$
            contains no two-sided projective planes) this must be as a
            sphere.  Flattening the bigon then has the effect of filling
            this boundary sphere with a 3-ball, as illustrated in
            Figure~\ref{fig-bigoncap}.
        \end{itemize}

        \begin{figure}[tb]
            \centering
            \includegraphics[scale=0.7]{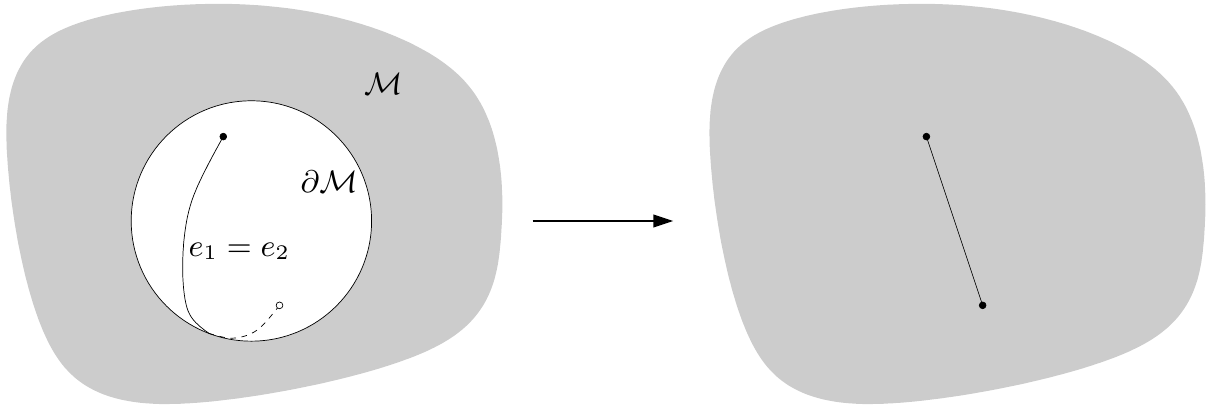}
            \caption{Flattening a bigon that forms a sphere on the boundary}
            \label{fig-bigoncap}
        \end{figure}

        \item If the bigon does not lie in the boundary and
        $e_1$ and $e_2$ are identified together so that the
        bigon forms a sphere,
        then flattening the bigon has the effect of slicing $\mfd$
        open along this sphere and filling the resulting boundary
        spheres with 3-balls, as illustrated in
        Figure~\ref{fig-bigonsphere}.  Note that this is true even if
        the edges and/or vertices of the bigon \emph{do} lie in the
        boundary, as illustrated in Figure~\ref{fig-bigonsphere-bdry}.

        \begin{figure}[tb]
            \centering
            \includegraphics[scale=0.7]{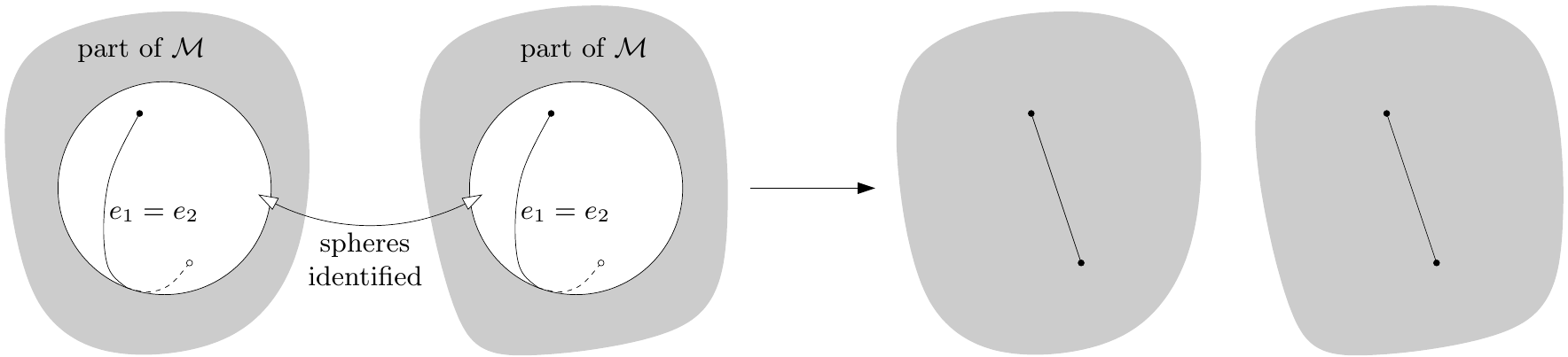}
            \caption{Flattening a bigon that forms an internal sphere}
            \label{fig-bigonsphere}
        \end{figure}

        \begin{figure}[tb]
            \centering
            \includegraphics[scale=0.7]{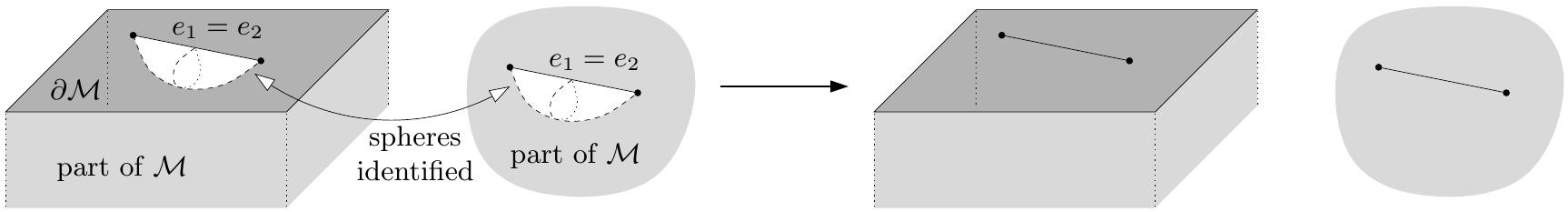}
            \caption{Flattening a bigon that forms a sphere touching the
                boundary}
            \label{fig-bigonsphere-bdry}
        \end{figure}

        \item If the bigon does not lie in the boundary and
        $e_1$ and $e_2$ are identified together so that the
        bigon forms a projective plane, then (by our initial assumptions)
        this must be a one-sided projective plane $P$.
        We can view the move in two stages: (i)~unglue the two cells on
        either side of the bigon (since this gluing will be lost after
        the move anyway), which yields two boundary bigons; and then
        (ii)~flatten these two boundary bigons to edges.

        If the entire bigon is internal to the manifold then,
        topologically, stage~(i) cuts along a sphere surrounding
        $P$ and discards the connected component containing $P$, and
        then stage~(ii) fills the remaining sphere boundary with a 3-ball.
        That is, we remove a single $\R P^3$ summand from the connected
        sum decomposition of $\mfd$.

        If the edge and/or vertex of this bigon lies in the
        boundary of $\mfd$ then stage~(i) may introduce temporary
        anomalies (such as punctures in vertex linking discs),
        but stage~(ii) fixes these so that, like the sphere case before,
        the full move has the same net topological effect regardless of
        whether the bigon is internal.

        Note that this operation does not create an invalid edge, even
        though the original edges $e_1,e_2$ that surround the bigon are
        identified in reverse.  This is because the projective plane is
        one-sided, and so stage~(i) creates a single sphere boundary
        formed from two bigons (the double cover of the original
        projective plane), which then allows us to flatten both boundary bigons
        without problems.  If the projective plane were two-sided then
        we would indeed create two invalid edges; we return to this
        possibility in Lemma~\ref{l-crush-invalid}.

        \item If the bigon does not lie in the boundary, and
        $e_1$ and $e_2$ are not identified but both $e_1$ and $e_2$
        \emph{do} lie in the boundary,
        then flattening the bigon has the effect of slicing
        $\mfd$ open along the corresponding disc, as illustrated in
        Figure~\ref{fig-bigonbdry}.

        \begin{figure}[tb]
            \centering
            \includegraphics[scale=0.9]{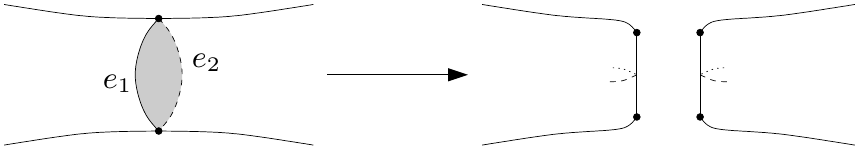}
            \caption{Flattening a bigon whose edges encircle the boundary}
            \label{fig-bigonbdry}
        \end{figure}

        \item Otherwise $e_1$ and $e_2$ are not identified (so their
        relative interiors are disjoint) and not both boundary, and so
        flattening them together does not change the
        underlying 3-manifold.
        \qedhere
    \end{itemize}
\end{proof}

Now that we understand the possible behaviour of each atomic move, we can
aggregate this information to understand the Jaco-Rubinstein crushing
procedure as a whole.  In the following result
we do this for arbitrary compact manifolds,
thereby generalising Theorem~\ref{t-jrcrushor} to both
orientable and non-orientable settings.  As in the previous lemma,
we exclude two-sided projective planes (which can lead to invalid edges);
however, even this exclusion can be partially overcome as we see later in
Section~\ref{s-app}.

\begin{corollary} \label{c-jrcrush}
    Let $\tri$ be a generalised triangulation of a compact
    3-manifold $\mfd$ (with or without boundary) that contains no
    two-sided projective planes, and let $S$ be a normal sphere or
    disc in $\tri$.  Then, if we destructively crush $S$ using the
    Jaco-Rubinstein procedure, we obtain a valid triangulation
    $\trijr$ whose underlying 3-manifold $\mfdjr$ is obtained from
    $\mfd$ by zero or more of the following operations:
    \begin{itemize}
        \item undoing connected sums, i.e.,
        replacing some intermediate manifold $\mfd'$ with the disjoint union
        $\mfd'_1 \cup \mfd'_2$, where $\mfd' = \mfd'_1\ \#\ \mfd'_2$;
        \item cutting open along properly embedded discs;
        \item filling boundary spheres with 3-balls;
        \item deleting 3-ball, 3-sphere, $\R P^3$, $L_{3,1}$,
        $S^2 \times S^1$ or twisted $S^2 \twisted S^1$ components.
    \end{itemize}
\end{corollary}

\begin{proof}
    This follows immediately from Lemmata~\ref{l-crush} and~\ref{l-atomic},
    and the fact that the \emph{non-destructive} act of crushing $S$ to
    a point (which precedes the sequence of atomic moves) has
    the effect of either undoing a connected sum (if $S$ is a separating
    sphere), removing an $S^2 \times S^1$ or twisted $S^2 \twisted S^1$
    summand from the connected sum decomposition
    (if $S$ is a non-separating sphere),
    or cutting along a properly embedded disc (if $S$ is a disc).
\end{proof}


\section{Non-orientable prime decomposition} \label{s-app}

We now give an application of our earlier results:
a modern approach to prime decomposition of non-orientable
manifolds based on the crushing process.

In 1995, Jaco and Tollefson described an algorithm that,
given a closed 3-manifold triangulation $\tri$, decomposes the
underlying manifold into a connected sum of prime manifolds
\cite{jaco95-algorithms-decomposition}.  In essence, it involves the
following steps:
\begin{enumerate}
    \item Enumerate all \emph{vertex normal spheres} in $\tri$.
    These are normal spheres that are represented by extreme rays of
    a high-dimensional polyhedral cone derived from the triangulation $\tri$;
    see \cite{jaco95-algorithms-decomposition} for a precise definition.%
    \label{en-jt-enum}
    \item Convert these into a (possibly much larger) collection of
    pairwise disjoint embedded spheres in $\tri$ using an intricate series of
    cut-and-paste operations.%
    \label{en-jt-paste}
    \item Cut $\tri$ open along these embedded spheres, retriangulate,
    and fill the boundaries with balls to obtain the final list of
    irreducible summands.%
    \label{en-jt-retriangulate}
\end{enumerate}

Despite its theoretical importance, the Jaco-Tollefson algorithm is both
slow and complex.  Step~\ref{en-jt-enum} requires us to enumerate
\emph{all} vertex normal spheres (of which there could be exponentially
many), which prevents us from using
highly effective optimisations based on linear programming
\cite{burton12-unknot}.  Step~\ref{en-jt-paste} is extremely complex
to implement, and could significantly increase the number of
spheres under consideration (which is already exponential in the
number of tetrahedra).
Likewise, the cut-open-and-retriangulate operation of
step~\ref{en-jt-retriangulate} is highly intricate to implement, and could
vastly increase the number of tetrahedra in the final collection of
triangulated summands.

For orientable triangulations, the Jaco-Rubinstein theory of 0-efficiency
from 2003 simplified this algorithm enormously
\cite{jaco03-0-efficiency}.
In brief, the new procedure is:
\begin{enumerate}
    \item Locate \emph{any} non-trivial normal sphere in the
    triangulation, and destructively crush this to obtain a new
    (possibly disconnected) triangulation with strictly fewer
    tetrahedra.  Repeat this step for as long as a non-trivial normal
    sphere can be found.

    \item Once no non-trivial normal spheres exist (i.e., the
    remaining triangulation is 0-efficient), each connected component of
    the triangulation represents a single prime summand.  There
    may be additional ``missing'' summands that were lost, but these can
    be reconstructed by tracking changes in homology.
\end{enumerate}

This 0-efficiency-based algorithm is much faster
(though still exponential time), and is significantly cleaner to
implement.
Moreover, it becomes far simpler to
detect and discard trivial 3-sphere summands, since
3-sphere recognition is significantly less demanding for 0-efficient
triangulations than for general inputs \cite{jaco03-0-efficiency}.
Historically this algorithm was the turning point at which
prime decomposition first became practical, and in 2004
it became the foundation for the
first real software implementation \cite{burton04-regina}.

For non-orientable triangulations, the state of the art remains the
original Jaco-Tollefson algorithm, which has still never been
implemented due to the speed and intricacy reasons outlined above.
Here we now use the results of Section~\ref{s-crush} to develop a
fast and simple prime decomposition algorithm for
non-orientable triangulations, based on 0-efficiency and
the Jaco-Rubinstein crushing process.

In this setting there is a major complication: for non-orientable manifolds,
the Jaco-Rubinstein crushing process might leave us with an
\emph{invalid triangulation}, where some edge is identified with
itself in reverse.  As noted in the detailed proof of
Lemma~\ref{l-atomic}, this can only occur if the triangulation contains
an embedded two-sided projective plane.

We employ a ``permissive'' strategy for dealing with this complication:
we run the algorithm regardless of whether there might be problems,
and after it finishes we test whether anything went wrong by looking
for invalid edges (an easy test to perform).
This permissive approach has two benefits:
\begin{itemize}
    \item There are \emph{no onerous preconditions} to test before we
    run the algorithm.\footnote{%
        The absence of two-sided projective planes is an ``onerous''
        precondition, in the sense that there is no algorithm
        known at present that can test this in polynomial time.}
    Instead we can start the algorithm immediately, oblivious to whether
    there is an embedded two-sided projective plane or not.

    \item The algorithm \emph{might still succeed} even if there is an
    embedded two-sided projective plane: if we ``get lucky'' and do not
    create an invalid edge, we still guarantee correctness.
    If we are unlucky and we do create an invalid edge during some
    atomic move,
    we prove that this can be detected
    after the fact, once the crushing process is complete.
\end{itemize}

We begin this section with Lemma~\ref{l-crush-invalid},
a non-orientable extension to Corollary~\ref{c-jrcrush} that uses the
crushing lemma to identify the possible consequences of crushing
in the presence of two-sided projective planes.
In the proof we take care to
ensure that, if an atomic move ever creates an invalid edge, then any
subsequent atomic moves \emph{preserve} the existence of invalid edges.
We follow this with the full prime decomposition
algorithm, as detailed in Algorithm~\ref{a-connsum}.

\begin{lemma} \label{l-crush-invalid}
    Let $\tri$ be a generalised triangulation of any closed compact
    3-manifold $\mfd$, and let $S$ be a normal sphere in $\tri$.
    Then, if we destructively crush $S$ using the Jaco-Rubinstein
    procedure, one of the following things happens:
    \begin{enumerate}
        \item we obtain an invalid triangulation, in which some edge is
        identified with itself in reverse;%
        \label{en-pp-invalid}
        \item we obtain a valid triangulation $\trijr$ whose underlying
        3-manifold $\mfdjr$ is obtained from $\mfd$ by zero or more of
        the following operations:
        \begin{itemize}
            \item undoing connected sums, i.e.,
            replacing some intermediate manifold $\mfd'$ with the disjoint union
            $\mfd'_1 \cup \mfd'_2$, where $\mfd' = \mfd'_1\ \#\ \mfd'_2$;
            \item deleting 3-sphere, $\R P^3$, $L_{3,1}$, $S^2 \times S^1$
            or twisted $S^2 \twisted S^1$ components.
        \end{itemize}%
        \label{en-pp-standard}
    \end{enumerate}
    Moreover, (\ref{en-pp-invalid}) can only occur if
    $\mfd$ contains an embedded
    two-sided projective plane.
\end{lemma}

\begin{proof}
    If $\mfd$ does not contain an embedded two-sided projective plane
    then this is a special case of Corollary~\ref{c-jrcrush}
    (restricted to closed manifolds),
    and it is immediate that we obtain outcome (\ref{en-pp-standard})
    above.

    If we do allow $\mfd$ to contain embedded two-sided projective
    planes, we must examine how this affects our three atomic
    operations from Lemma~\ref{l-crush}: flattening triangular pillows,
    flattening bigonal pillows, and flattening bigon faces.
    In essence, we find that things behave well around the vertices, but
    we may create \emph{invalid edges} (edges identified with
    themselves in reverse).  This occurs precisely when we flatten a bigon
    face whose edges are identified to form a two-sided projective
    plane; we describe this operation in more detail shortly.
    Note that both endpoints of any invalid edge $e$
    must be identified; that is, $e$ is incident to one and only one vertex.

    Before proceeding, we observe that
    \emph{every vertex link remains a 2-sphere under all three
    atomic operations.}
    This is because (i)~flattening a triangular or bigonal pillow removes
    an entire connected component, and so cannot \emph{introduce}
    a non-spherical vertex link; and
    (ii)~flattening a bigon face effectively collapses two curves to a
    point in the vertex links, which might split a 2-sphere link into
    multiple 2-spheres, but which cannot introduce \emph{new}
    genus or orientation-reversing paths.\footnote{%
        If a 2-sphere link does split into multiple 2-spheres,
        this means that a vertex of the cell decomposition splits
        into multiple vertices.  This happens, for instance, when the
        bigon forms an embedded sphere, and flattening the bigon
        effectively undoes a connected sum.}

    As we consider each atomic move, we show that either the move is
    consistent with outcome~(\ref{en-pp-standard}) above, or else we
    carefully study how it affects the number and location of any
    invalid edges.  By the ``location'' of an invalid edge,
    we mean the \emph{incident vertex} (of which there is only one,
    as noted above).

    Our first step is to remove triangular pillows from consideration.
    Here we observe that nothing ``goes wrong'', in that
    \emph{flattening a triangular pillow is always consistent
    with outcome~(\ref{en-pp-standard}) above, 
    and never changes the number or location of any invalid edges}.
    Let $F_1,F_2$ be the two triangular faces bounding such a pillow:
    \begin{itemize}
        \item If $F_1$ and $F_2$ are not identified then flattening the
        pillow does not change the underlying 3-manifold,
        and although the pillow may be surrounded by one or more invalid
        edges, flattening the pillow does not change their number or
        location.
        \item If $F_1$ and $F_2$ are identified in an
        orientation-preserving fashion, then the pillow is a single
        $S^3$ or $L_{3,1}$ component (with no invalid edges),
        and flattening the pillow simply removes this component.
        \item If $F_1$ and $F_2$ are identified in an
        orientation-reversing fashion then one of the vertices
        of the pillow will have a
        non-orientable link, which we know from above can never occur.
    \end{itemize}

    We next ``limit the damage'' that can occur from bigonal pillows.
    Here we show that
    \emph{flattening a bigonal pillow is either consistent with
    outcome~(\ref{en-pp-standard}) above with no change in the
    number or location of invalid edges, or else it
    removes precisely two invalid edges, both of which are
    incident to the same vertex}.
    Again, let $F_1,F_2$ be the two bigon faces bounding such a pillow:
    \begin{itemize}
        \item If $F_1$ and $F_2$ are not identified then flattening the
        pillow does not change the underlying 3-manifold,
        and again does not affect the number or location of any invalid edges.
        \item If $F_1$ and $F_2$ are identified in an
        orientation-preserving fashion, then the pillow is a single
        $S^3$ or $\R P^3$ component (with no invalid edges),
        and as before flattening the pillow simply removes this component.
        \item Suppose that $F_1$ and $F_2$ are identified in an
        orientation-reversing fashion.  There are two choices of
        identification:
        (i)~one that maps each edge of the bigon to the other,
        and (ii)~one that maps each edge to itself in reverse.
        The first option yields non-orientable vertex links, and so
        cannot occur.  The second option yields precisely two
        invalid edges, both incident to the same vertex, and the
        flattening operation removes the entire pillow along
        with both of these invalid edges.
    \end{itemize}

    This leaves us with the most problematic move: flattening a bigon
    face.  Let $e_1,e_2$ be the two edges bounding such a face.
    Here one of several things can occur:
    \begin{itemize}
        \item If $e_1$ and $e_2$ are both valid edges,
        and if they are \emph{not} identified in a way that makes
        the bigon form a two-sided
        projective plane, then things work exactly as in the proof of
        Lemma~\ref{l-atomic}: we might slice $\mfd$ open along a sphere
        and fill the resulting boundaries with 3-balls, or we might
        remove a single $\R P^2$ summand from the connected sum
        decomposition of $\mfd$, or we might simply leave the manifold
        unchanged.  All of these possibilities are consistent
        with outcome~(\ref{en-pp-standard}) above.

        Here we do not change the number of invalid edges, but we might
        change their locations.  Specifically, we might split a vertex $V$
        into multiple vertices $V_1,V_2$ (e.g., when undoing a connected
        sum), whereupon the invalid edges that
        touched $V$ will become distributed amongst $V_1$ and $V_2$ in
        some manner.  Note that we might even perform more than one such split.

        \begin{figure}[tb]
            \centering
            \includegraphics[scale=0.9]{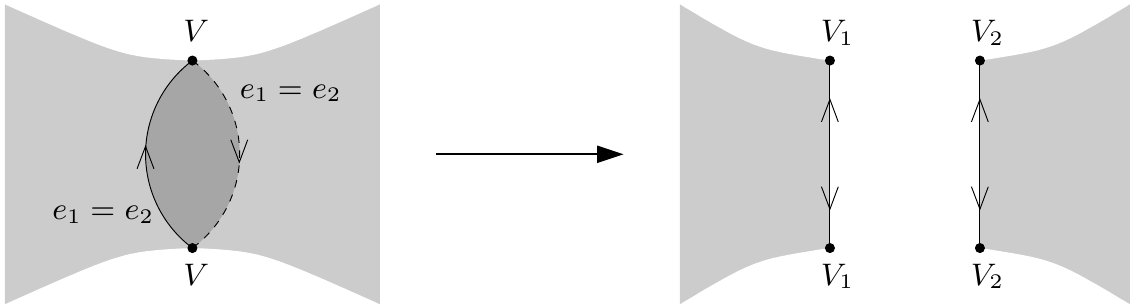}
            \caption{Flattening a bigon that forms a two-sided projective plane}
            \label{fig-createinvalid}
        \end{figure}

        \item If $e_1$ and $e_2$ are both valid edges,
        and if they \emph{are} identified in a way that makes
        the bigon form a two-sided projective plane, then the move has
        the following effect.  We cut the manifold open along this
        projective plane, and then flatten each of the two $\R P^2$ boundary
        components to a single invalid edge, as illustrated in
        Figure~\ref{fig-createinvalid}.  Note that the midpoint of each
        invalid edge has a small regular neighbourhood bounded by
        $\R P^2$ (not $S^2$ as usual).

        The outcome is that we create two new invalid edges.
        The locations change as follows: the original vertex $V$ splits
        into two vertices $V_1,V_2$; any previous invalid edges that touched
        $V$ become distributed amongst $V_1$ and $V_2$ in some manner;
        and then each of $V_1$ and $V_2$ acquires one of the new invalid
        edges that is created by the move.

        \item If one edge (say $e_1$) is valid but the other (say $e_2$)
        is invalid, then the move merges these together into a single invalid
        edge that is incident with the same vertex as the original $e_2$.
        That is, both the number and location of all invalid edges
        stays the same.

        \item If $e_1$ and $e_2$ are both invalid, then we note that they
        must both be incident to the same vertex.
        If $e_1$ and $e_2$ are invalid and \emph{not} identified, then
        the move merges $e_1$ and $e_2$ together into a single
        \emph{valid} edge (the orientation twists around $e_1$ and $e_2$
        effectively cancel each other when the edges are merged).
        The result is that we lose precisely two invalid edges, both of
        which were incident to the same vertex.

        \item If $e_1$ and $e_2$ are both invalid,
        and if they \emph{are} identified, then the bigon must
        form a sphere whose edge is identified with itself
        in reverse due to a twist on one side of the sphere,
        as illustrated in Figure~\ref{fig-crushinvalid}.
        The move essentially (i)~cuts the manifold open along this sphere,
        leaving an $\R P^2$ boundary on the side with the twist and a
        sphere boundary on the other side, and then (ii)~flattens each of
        these boundaries to a single edge, creating an invalid edge on
        the $\R P^2$ side.  The full move is illustrated in
        Figure~\ref{fig-crushinvalid}.

        The outcome is that the total number of invalid edges is
        preserved (including our original edge $e_1=e_2$, which becomes
        the flattened $\R P^2$ boundary).
        Regarding locations: again the original vertex $V$ splits into
        multiple vertices $V_1,V_2,V_3$, whereupon the invalid edges originally
        incident to $V$ become distributed amongst $V_1,V_2,V_3$ in some manner.

        \begin{figure}[tb]
            \centering
            \includegraphics[scale=0.75]{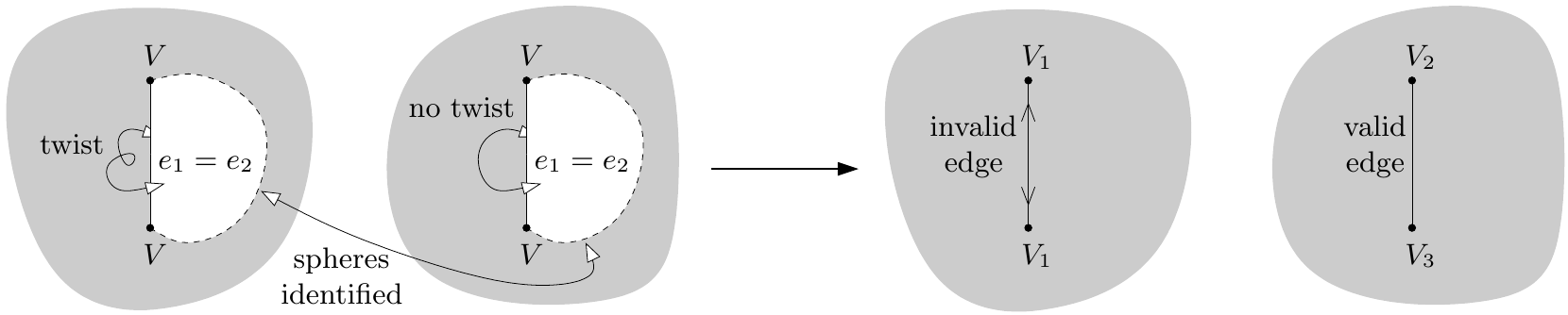}
            \caption{Flattening a bigon that joins an invalid edge to itself}
            \label{fig-crushinvalid}
        \end{figure}
    \end{itemize}

    In summary: taking into account all three atomic moves, we find
    that if no invalid edges are ever created, then we obtain
    outcome~(\ref{en-pp-standard}) from the lemma statement.
    Our final goal is to show that, if invalid edges \emph{are}
    ever created, that at least one of them survives to the end of the
    crushing procedure; that is, we obtain outcome~(\ref{en-pp-invalid}).

    This is now just a matter of parity.
    Define an \emph{odd or even vertex} to be one that is incident with an
    odd or even number of invalid edges respectively.
    The first time we create a pair of invalid edges,
    these are incident with different vertices; that is, we create two
    odd vertices.
    It is now simple to see that none of the moves can ever reduce the
    number of odd vertices:
    \begin{itemize}
        \item whenever we split a vertex $V$ and redistribute its
        incident edges amongst the resulting vertices $V_1,V_2,\ldots$,
        if the original vertex $V$ was odd then one of
        the resulting vertices $V_i$ must be odd also;
        \item if we ever create a new pair of invalid edges from a
        two-sided projective plane incident to some original odd vertex
        $V$, then the two resulting vertices $V_1$ and $V_2$ must have
        different parities, and so one of these must be odd also;
        \item whenever we lose a pair of invalid edges (either through
        flattening a bigon or flattening a bigonal pillow), these invalid
        edges are always incident to the same vertex and so parities are
        preserved.
    \end{itemize}

    It follows that, if we ever create an invalid edge at any stage,
    then we will have odd vertices remaining at the end of the crushing
    process; that is, we will have invalid edges as described in
    outcome~(\ref{en-pp-invalid}).
\end{proof}

We can now package the results of Lemma~\ref{l-crush-invalid}
into a general algorithm for computing the prime decomposition
of a triangulated 3-manifold, either orientable or non-orientable.
The structure of the algorithm follows the modern
``ready to implement'' framework presented in \cite{burton13-regina}
for the orientable case.
The process can be further improved by simplifying triangulations at
key stages of the algorithm; we omit this here, but details
can be found in \cite{burton13-regina}.

\begin{algorithm}[Prime decomposition] \label{a-connsum}
    Given an input triangulation $\tri$ of any closed connected 3-manifold
    $\mfd$, the following algorithm will either decompose $\mfd$ into a
    connected sum of prime manifolds, or else prove that
    $\mfd$ contains an embedded two-sided projective plane.
    \begin{enumerate}
        \item \label{step-homology}
        Compute the first homology of $\tri$, and let
        $r$, $t_2$ and $t_3$ denote the $\Z$~rank,
        $\Z_2$~rank and $\Z_3$~rank respectively.
        \item Create an input list $\mathcal{L}$ of triangulations to
        process, initially containing just $\tri$, and an output list
        $\mathcal{O}$ of prime summands, initially empty. \\
        While $\mathcal{L}$ is non-empty:
        \begin{itemize}
            \item Let $\mathcal{N}$ be the next triangulation in
            the list $\mathcal{L}$.  Remove $\mathcal{N}$ from
            $\mathcal{L}$, and test whether $\mathcal{N}$
            has a non-trivial normal sphere $F$.%
            \label{step-sphere}
            \begin{itemize}
                \item
                If there is such a normal sphere,
                then perform the Jaco-Rubin\-stein crushing procedure on $F$.
                \begin{itemize}
                    \item If the resulting triangulation has an invalid
                    edge, then terminate with the statement that the
                    input manifold contains an embedded two-sided
                    projective plane.
                    \item If the resulting triangulation has no invalid
                    edges, then add each connected component
                    of the resulting triangulation
                    back into the list $\mathcal{L}$.
                \end{itemize}
                \item
                If there is no such normal sphere, then
                append $\mathcal{N}$ to the output list $\mathcal{O}$.
            \end{itemize}
        \end{itemize}
        \item Compute the first homology of each triangulation in
        the output list $\mathcal{O}$, and let $r'$, $t_2'$ and
        $t_3'$ denote the sums of the $\Z$~ranks, $\Z_2$~ranks and $\Z_3$~ranks
        respectively.
        \item Append $(t_2-t_2')$ copies of $\R P^3$ and
        $(t_3-t_3')$ copies of $L_{3,1}$ to $\mathcal{O}$.
        If the input triangulation was orientable, append
        $(r-r')$ copies of $S^2 \times S^1$ to $\mathcal{O}$,
        and otherwise append $(r-r')$ copies of the twisted product
        $S^2 \twisted S^1$ to $\mathcal{O}$.
    \end{enumerate}
    If we did not terminate earlier due to an invalid edge,
    then the final output list $\mathcal{O}$ will contain
    a collection of triangulated prime manifolds
    $\mathcal{O}_1,\ldots,\mathcal{O}_k$
    for which the original manifold $\mfd$ can be expressed as the 
    connected sum $\mathcal{O}_1 \# \mathcal{O}_2 \# \ldots \#
    \mathcal{O}_k$.
\end{algorithm}

The correctness of this algorithm follows immediately from
Lemma~\ref{l-crush-invalid}
(the changes in $\Z$, $\Z_2$ and $\Z_3$ ranks indicate the number of
$S^2 \times S^1$ or $S^2 \twisted S^1$,
$\R P^3$, and $L_{3,1}$ summands that were lost respectively).
If the input triangulation contains $n$ tetrahedra, it is clear that we
terminate after crushing at most $n$ normal spheres, since each crushing
operation strictly reduces the total number of tetrahedra in the input
list $\mathcal{L}$.  Note that some of the output manifolds
$\mathcal{O}_k$ might be trivial (i.e., redundant 3-sphere summands);
however, such trivial summands are easy to detect, as outlined below.

As presented above, this algorithm
gives the \emph{summands} in the connected sum
decomposition, but these may not uniquely define the original
manifold (since there can be orien\-ta\-tion-related decisions to make when
performing the
connected sum operation).  This can be resolved by tracking orientations
explicitly through the crushing process, a straightforward but
slightly messy enhancement to the algorithm that we do not describe
in detail here.

We finish this section with some implementation notes:
\begin{itemize}
    \item
    There are well-known polynomial-time procedures for computing homology in
    step~(\ref{step-homology}), based on Smith normal form; see
    \cite{donald91-homology,dumas03-homology} for examples.

    \item
    In step~(\ref{step-sphere}) we must locate a non-trivial normal sphere,
    if one exists.  Traditionally, one does this by enumerating all
    \emph{quadrilateral vertex normal surfaces};
    see \cite{burton13-regina} for details on what this means and why
    it works.
    A newer (and experimentally much faster) alternative is to
    make a targeted search for a normal sphere using branch-and-bound
    techniques from combinatorial optimisation;
    see \cite{burton12-unknot} for details.

    \item
    It is easy to eliminate trivial 3-sphere summands from the output list.
    If an output triangulation $\mathcal{O}_i$
    has non-trivial homology then it is a non-trivial summand;
    otherwise $\mathcal{O}_i$ must be 0-efficient,
    whereupon Jaco and Rubinstein show that $\mathcal{O}_i$ is trivial
    if and only if (i)~it has more than one vertex, or
    (ii)~it contains an embedded \emph{almost normal sphere}.
    See \cite{jaco03-0-efficiency} for details on almost normal spheres,
    and see \cite{burton10-quadoct,burton12-unknot} for fast
    algorithms for detecting them.
\end{itemize}


\section{Minimal triangulations} \label{s-app-minimal}

Our final application uses crushing techniques to study the
combinatorics of minimal triangulations of both orientable and
non-orientable manifolds.  A \emph{minimal triangulation} of a
3-manifold $\mfd$ is one that triangulates $\mfd$ using the fewest
possible tetrahedra.  Minimal triangulations are the focus of
many censuses of low-complexity 3-manifolds \cite{matveev03-algms};
in such censuses
it is common to focus on \emph{\ppirr} 3-manifolds, in which every embedded
2-sphere bounds a 3-ball, and in which there are no embedded
two-sided projective planes.

There are many simple combinatorial properties that a minimal triangulation
of a closed {\ppirr} manifold must satisfy (for instance, with a few
exceptions it must have just one vertex, and no low-degree edges).
Often these properties are proven using ad-hoc constructions that show how,
if a property is not satisfied, then the triangulation can be simplified
to use fewer tetrahedra.

As part of their study on 0-efficient triangulations
\cite{jaco03-0-efficiency}, Jaco and Rubinstein prove several of these
properties for orientable manifolds using 0-efficiency techniques.
Although these are in most cases rederivations of known results, their
proofs essentially \emph{unify} these results: instead of a collection of
ad-hoc constructions, they show how these results all follow immediately
from the observation that (modulo a few exceptions)
a minimal triangulation must be 0-efficient.

Here we extend these arguments to the non-orientable setting.  Again
most of the results we prove are already known; the purpose of this
section is to unify them as immediate consequences of 0-efficiency.
Much of this section follows Jaco and Rubinstein directly,
and so we keep the exposition brief and only present details where the
arguments diverge.

All of the results in this section rely on the following core
observation, which Jaco and Rubinstein prove for orientable manifolds
\cite{jaco03-0-efficiency}, and which we prove here in the general case.

\begin{theorem} \label{t-minimal}
    Let $\mfd$ be a closed {\ppirr} 3-manifold (which may be orientable
    or non-orientable) that is not $\R P^3$ or $L_{3,1}$.
    Then every minimal triangulation of $\mfd$ is 0-efficient.
\end{theorem}

\begin{proof}
    This follows immediately from Corollary~\ref{c-jrcrush}:
    if a triangulation of $\mfd$ is not 0-efficient,
    then performing the Jaco-Rubinstein crushing procedure on a normal
    2-sphere will produce a new triangulation of $\mfd$ with fewer
    tetrahedra.

    If $\mfd$ is the 3-sphere then crushing might remove $\mfd$
    entirely, but in this case we simply note that both minimal
    triangulations of $S^3$ (each with one tetrahedron) are 0-efficient,
    as observed earlier in \cite{jaco03-0-efficiency}.
\end{proof}

We now turn to proving various properties of minimal triangulations.
Our first result---that a minimal triangulation must have one
vertex---was originally proven by Matveev for orientable manifolds
\cite{matveev90-complexity} and Martelli and Petronio in the general case
\cite{martelli02-decomp}, both working in the dual setting of special spines.
Jaco and Rubinstein rederive this for orientable manifolds
as a simple corollary of 0-efficiency \cite{jaco03-0-efficiency},
and here we observe that their argument translates directly to the
non-orientable case.

\begin{corollary}
    Let $\tri$ be a minimal triangulation of a closed
    {\ppirr} 3-manifold (which may be orientable or non-orientable)
    that is not $S^3$, $\R P^3$ or $L_{3,1}$.
    Then $\tri$ contains precisely one vertex.
\end{corollary}

\begin{proof}
    We first show that any 0-efficient triangulation with $\geq 2$ vertices
    must represent the 3-sphere.  This is a direct copy of Jaco and
    Rubinstein's barrier surface argument from
    \cite[Proposition~5.1]{jaco03-0-efficiency}, and we do not repeat it
    here.  The essential details are as follows.
    
    If $\tri$ has two distinct vertices, then it has some edge $e$ that
    joins them.  Place a sphere $S$ around $e$, and attempt to normalise
    $S$.  Since the edge $e$ acts as a barrier to normalisation, and
    since the triangulation is 0-efficient, $S$ must normalise to a
    (possibly empty) union of vertex linking spheres disjoint from $e$.
    In this case we see that the underlying 3-manifold is obtained from
    a punctured 3-ball (containing $e$) whose boundary spheres
    (the vertex links) are filled with 3-balls on the other side
    (where the vertices lie); that is, the manifold is $S^3$.

    The result is now a simple application of Theorem~\ref{t-minimal}:
    since $\tri$ is minimal it must be 0-efficient, and by the argument
    above it cannot have $\geq 2$ vertices.
\end{proof}

We move now to properties of the edges and faces of a minimal
triangulation.  Here we prove that no edge can bound an embedded disc,
a new result for non-orientable manifolds
that extends the orientable result of Jaco and Rubinstein
\cite{jaco03-0-efficiency}.
This allows us to unify the observations that
no edge can have degree one
(previously shown by Matveev \cite{matveev98-or6} and the author
\cite{burton04-facegraphs} for orientable versus non-orientable manifolds),
and that no face can form a cone
(previously shown by Martelli and Petronio \cite{martelli01-or9}
and the author \cite{burton04-facegraphs} for orientable versus
non-orientable manifolds).

\begin{corollary}
    Let $\tri$ be a minimal triangulation of a closed
    {\ppirr} 3-manifold $\mfd$ (which may be orientable
    or non-orientable) that is not $S^3$, $\R P^3$ or $L_{3,1}$.  Then:
    \begin{itemize}
        \item no edge of $\tri$ bounds an embedded disc in $\mfd$;
        \item no edge of $\tri$ has degree one;
        \item no face $F$ of $\tri$ has two of its edges identified
        to form a cone (regardless of whether all three edges of
        $\Delta$ are identified or not).
    \end{itemize}
\end{corollary}

\begin{proof}
    We first observe that, by Theorem~\ref{t-minimal}, $\tri$ must be
    0-efficient.  From here we follow directly the arguments
    that Jaco and Rubinstein use in the orientable case.

    For the first result we use another barrier surface argument;
    this follows the proof of
    \cite[Proposition~5.3]{jaco03-0-efficiency}, and again we do not repeat
    the details here.  In essence, if some edge $e$ bounds an embedded
    disc, then we place a sphere $S$ around this disc and normalise.
    Again the edge $e$ acts as a barrier to normalisation, and so
    by 0-efficiency $S$ must normalise to a (possibly empty) union of
    vertex linking spheres disjoint from $e$, whereupon the same argument as
    before shows that $\mfd$ must be $S^3$.

    For the second result we follow \cite[Corollary~5.4]{jaco03-0-efficiency}.
    If some edge $e$ in some tetrahedron $\Delta$ has degree one,
    then the opposite edge of $\Delta$ becomes a loop that bounds an
    embedded disc slicing through both $\Delta$ and $e$, contradicting
    the first result.

    For the third result we generalise the argument of
    \cite[Corollary~5.4]{jaco03-0-efficiency} (which works in the
    more restricted setting where $\mfd$ is orientable and the third
    edge of the cone is not identified with the others).
    Suppose the face $F$ has two edges identified to form a cone, and
    denote these two identified edges by $e$.  If we denote the third
    edge of $F$ by $e'$, then we observe that the cone itself forms a disc
    bounded by $e'$.  We can isotope (slide) the relative interior of this
    disc through the manifold $\mfd$ to make an embedded disc as follows:
    \begin{itemize}
        \item If edge $e'$ is not identified with $e$, then the only
        way this cone might intersect itself is if the vertex at the centre
        of the cone is identified with the vertex on the boundary,
        as illustrated in Figure~\ref{fig-disc-vertex}.
        Here we simply push the relative interior of the disc
        off the boundary vertex, as illustrated.

        \begin{figure}[tb]
        \centering
        \subfigure[Pushing away from the boundary vertex]{%
            \label{fig-disc-vertex}%
            \includegraphics[scale=0.75]{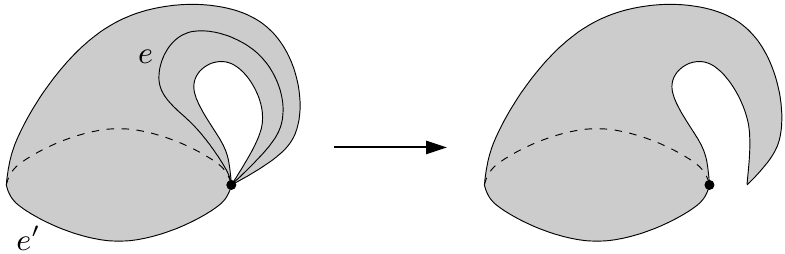}}
        \hspace{0.5cm}
        \subfigure[Pushing away from the boundary edge]{%
            \label{fig-disc-edge}%
            \includegraphics[scale=0.75]{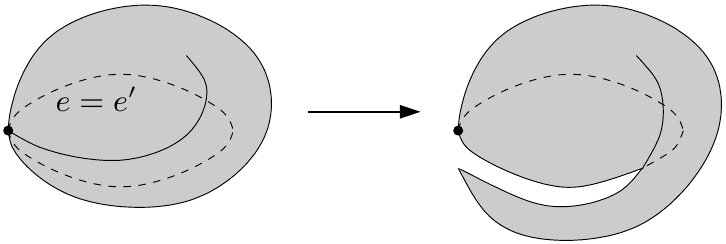}}
        \caption{Pushing the interior of a cone away from its boundary}
        \end{figure}

        \item If edges $e'$ and $e$ are identified, then the
        edge on the boundary of the cone is identified with the
        ``radius'' edge leading to the apex of the cone, as illustrated
        in Figure~\ref{fig-disc-edge}.  Once again we make the
        disc embedded, this time by pushing the interior
        of the disc away from the boundary edge, as illustrated.
    \end{itemize}
    In either case we obtain an embedded disc bounded by the edge $e'$,
    which contradicts the first result above.
\end{proof}
%
%

\bibliographystyle{amsplain}
\bibliography{pure}

\end{document}